\documentclass[12pt]{amsart}
\usepackage{amssymb,amsmath}
\usepackage{geometry}    
\usepackage{mathrsfs}

\usepackage{vmargin}
\setmargrb{1in}{1in}{1in}{1in}

\newtheorem{theorem}{Theorem}[section]
\newtheorem{lemma}[theorem]{Lemma}

\newtheorem{corollary}[theorem]{Corollary}
\theoremstyle{definition}
\newtheorem{definition}[theorem]{Definition}
\newtheorem{remark}[theorem]{Remark}
\newtheorem{example}[theorem]{Example}
\newtheorem{conjecture}[theorem]{Conjecture}

\begin{document}

\title{Rational approximation to surfaces defined by polynomials in one variable}

\author{Johannes Schleischitz} 

\address{Institute of Mathematics, Boku Vienna, Austria  \\ 
johannes.schleischitz@boku.ac.at}


\begin{abstract}
We study the rational approximation properties of special manifolds defined by a set of
polynomials with rational coefficients. Mostly we will assume the case of all polynomials 
to depend on only one variable. In this case the manifold can be viewed as a Cartesian product of polynomial curves
and it is possible to generalize recent results concerning such curves with similar concepts. There is
hope that the method leads to insights on how to treat more general manifolds defined
by arbitrary polynomials with rational coefficients.  
\end{abstract}

\maketitle

{\footnotesize{Supported by the Austrian Science Fund FWF grant P24828.} \\

Math Subject Classification 2010: 11J13, 11J83}

\vspace{4mm}

\section{Definitions and known results}

\subsection{Definitions}

We investigate the set of points $\underline{\zeta}=(\zeta_{1},\ldots,\zeta_{k})$
on special manifolds in $\mathbb{R}^{k}$ which are 
simultaneously approximable to some given degree by rational vectors.
By this we mean that $\max_{1\leq j\leq k} \vert \zeta_{j}-p_{j}/q\vert$ should be small
in terms of (negative powers of) $q$ for certain $(p_{1},\ldots,p_{k},q)\in{\mathbb{Z}^{k+1}}$ 
with arbitrarily large $q$. 
It will be more convenient to work with the equivalent problem of minimizing the expression
$\max_{1\leq j\leq k} \Vert q\zeta_{j}\Vert$ for certain arbitrarily large $q$, 
where $\Vert.\Vert$ denotes the distance of a real number to the nearest integer. 
More precisely, we introduce the following notation for the set of points approximable to a given degree, 
first without the restriction to some submanifold of $\mathbb{R}^{k}$, already used in~\cite{schlei2}.

\begin{definition}
Let $k\geq 1$ be an integer and $\lambda>0$ a parameter. 
Let $\mathscr{H}^{k}_{\lambda}$ be the set of 
$\underline{\zeta}=(\zeta_{1},\ldots,\zeta_{k})\in{\mathbb{R}^{k}}$ such that for all fixed
$\epsilon>0$ the estimate
\[
\max_{1\leq j\leq k} \Vert q\zeta_{j}\Vert \leq q^{-\lambda+\epsilon}
\]
has arbitrarily large integer solutions $q$.
\end{definition}

By Dirichlet's Theorem, see for example~\cite[p. 177-192]{lang},
we have $\mathscr{H}^{k}_{\lambda}=\mathbb{R}^{k}$ for $\lambda\leq 1/k$.
Concerning larger parameters it is known thanks to Khintchine~\cite{khint} that 
any $\epsilon>0$ the set of points in $\mathscr{H}^{s}_{1/k+\epsilon}$
has $k$-dimensional Lebesgue measure $0$. 
Jarn\'ik~\cite{jarnik} more generally established the Hausdorff dimension of the sets $\mathscr{H}^{k}_{\lambda}$.
Before we state his result we briefly recall the notion of Hausdorff dimension. For our purposes
it suffices to restrict to Euclidean spaces and dimension functions of the form $t\to t^{s}$, and we just refer 
to~\cite{berdod},\cite{falconer} for more general concepts and further details. 

\begin{definition}[Hausdorff dimension]
For a set $A\subseteq \mathbb{R}^{k}$ and any $\delta>0$ we 
define a $\delta$-cover of $A$ as a countable collection
of sets $B_{1},B_{2},\ldots$ such that each set $B_{i}\subseteq \mathbb{R}^{k}$ has diameter at most $\delta$
and the union of the $B_{i}$ contains $A$. For any $\delta$-cover and $s\geq 0$ consider  
the sum of $\rm{diam}(B_{i})^{s}$ over $i\geq 1$, 
and denote the infimum of these values over all $\delta$-coverings of $A$
by $I^{s}_{\delta}$. The limit of $I^{s}_{\delta}$ as $\delta\to 0$ exists as an element of $[0,\infty]$ 
and is called $s$-dimensional Hausdorff measure of $A$. The
supremum of $s\geq 0$ such that the $s$-dimensional Hausdorff measure of $A$ equals $\infty$ is called 
Hausdorff dimension of $A$ (and $0$ if no such $s$ exists, that is $A$ is finite).
\end{definition}

\begin{theorem}[Jarn\'ik] \label{jarnikchen}
For $k\geq 1$ an integer and $\lambda \geq 1/k$, the set 
$\mathscr{H}^{k}_{\lambda}$ has Hausdorff dimension $(k+1)(1+\lambda)^{-1}$.
\end{theorem}

We will be interested in the Hausdorff dimension of $M\cap \mathscr{H}^{k}_{\lambda}$
for manifolds $M$ defined by certain polynomials. The most general form of $M$ we will
consider is
\[
M=\left\{(x_{1},x_{2},\ldots,x_{s},P_{1}(x_{1},\ldots,x_{s}),\ldots,P_{r}(x_{1},\ldots,x_{s}))\in{\mathbb{R}^{k}}:
(x_{1},\ldots,x_{s})\in{\mathbb{R}^{s}}\right\}
\]
where $P_{j}$ are polynomials with rational coefficients. The metric theory of rational
approximation to certain manifolds of this type has been investigated by Budarina,
Dickinson and Levesley in~\cite{bu}, Budarina and Dickinson in~\cite{budi}
and recently the author~\cite{schlei2}. We will
quote some results from these papers in Section~\ref{schle}.
In this paper the focus is on the special case where $M$ is given in separate variables,
by which we mean each $P_{j}$ depends on only one variable. This restriction enables us to
carry out similar methods as in~\cite{schlei2}. We will obtain comparable, more general results.
However, apart from special cases, in this more general context we will not be able to determine 
the exact dimension of $M\cap \mathscr{H}^{k}_{\lambda}$ but only lower and upper bounds. 

We need projections at some places. 

\begin{definition}
For $s\geq 1$ and $r\geq 0$ integers let $\Pi_{s}: \mathbb{R}^{r+s}\to \mathbb{R}^{s}$  denote
the projection on the first $s$ coordinates, that is $\Pi_{s}(x_{1},\ldots,x_{r+s})=(x_{1},\ldots,x_{s})$.
\end{definition}

Recall that the image of a measurable set under Lipschitz maps, like projections (which are even contractions), 
of a set have at most the Hausdorff dimension of the original set. Hence, if a map and its inverse are both Lipschitz,
then it preserves dimensions.

\subsection{Results from~\cite{schlei2}} \label{schle}

We will apply the main results from~\cite{schlei2}, where the manifold is a curve defined by
polynomials in one variable.

\begin{theorem}[Schleischitz] \label{jox}
Let $k\geq 1$ be an integer and $\mathscr{C}$ be a curve given as 
\begin{equation} \label{eq:curve}
\mathscr{C}=\{(X,P_{2}(X),\ldots,P_{k}(X)): \; X\in{\mathbb{R}}\}, \qquad P_{j}\in{\mathbb{Q}[X]}, 
\end{equation}
where $P_{1}(X)=X$.
Assume the degree of $P_{j}$ is $d_{j}$ and they are labeled such that $1=d_{1}\leq d_{2}\leq \ldots\leq d_{k}$
and $d_{k}\geq 2$.
Let $t:=\max_{1\leq j\leq k-1} \{d_{j+1}-d_{j}\}\geq 1$ the diameter of $\mathscr{C}$. 
Then for any parameter $\lambda>t$ we have 
\begin{equation} \label{eq:neue}
\Pi_{1}(\mathscr{C}\cap \mathscr{H}^{k}_{\lambda})= \mathscr{H}^{1}_{d_{k}\lambda+d_{k}-1}. 
\end{equation}
\end{theorem}

Theorem~\ref{jox} in combination with Jarn\'ik~Theorem~\ref{jarnikchen}
led to the value $2d_{k}^{-1}/(1+\lambda)$ for the dimension 
of $\mathscr{C}\cap \mathscr{H}^{k}_{\lambda}$ for $\lambda>t$.
Essentially the same formula was established in~\cite{bu} for the (possibly) smaller range of 
parameters $\lambda\in{[d_{k}-1,\infty]}$. 
The results in~\cite{bu} were on the other hand slightly more general in the sense of permitting more
general dimension functions. However, our results in both~\cite{schlei2} and the present paper can be
generalized to this more general context rather straightforward. See also~\cite{schlei}
for the special case of the Veronese curve where the proof is less technical. It is worth
noting that the inclusion 
$\Pi_{1}(\mathscr{C}\cap \mathscr{H}^{k}_{\lambda})\supseteq \mathscr{H}^{1}_{d_{k}\lambda+d_{k}-1}$
holds for all parameters $\lambda\geq 1/k$, which is~\cite[Lemma~1.2]{schlei2}. We will generalize this
result in~Lemma~\ref{dadlemmane}. 

At places we will need to apply the following Lemma~\ref{lemma2}, which is~\cite[Lemma~3.3]{schlei2}, directly.
It was the key observation for the proof of Theorem~\ref{jox}.
As in~\cite{schlei2}, to avoid heavy notation in the formulation of Lemma~\ref{lemma2}, 
we prepone some definitions. 
For $\mathscr{C}$ as in \eqref{eq:curve} with polynomials $P_{j}\in{\mathbb{Z}[X]}$ of non-decreasing degrees 
$d_{j}$ as in Theorem~\ref{jox}, write
\begin{equation} \label{eq:polynome}
P_{j}(X)=c_{0,j}+c_{1,j}X+\cdots+c_{d_{j},j}X^{d_{j}}, \qquad c_{.,j}\in{\mathbb{Z}}, \quad 1\leq j\leq k.
\end{equation}
Moreover, for $\zeta\in{\mathbb{R}}$ we define
\begin{equation} \label{eq:stock}
\Delta=\Delta(\mathscr{C}):=\prod_{1\leq j\leq k} \vert c_{d_{j},j}\vert,  
\quad D=D(\mathscr{C}):=\Delta^{d_{k}}, \quad 
\Sigma(\mathscr{C},\zeta):=\max_{1\leq j\leq k}\max_{\vert z-\zeta\vert\leq 1/2} \vert P_{j}^{\prime}(z)\vert.
\end{equation}
Furthermore, for $x_{0}$ an integer variable that will appear in the lemma and $\Delta$ in \eqref{eq:stock}, let
\begin{equation} \label{eq:xeins}
x_{1}:=\frac{x_{0}}{(x_{0},\Delta)},
\end{equation}
where $(.,.)$ denotes the greatest common divisor.

\begin{lemma}[Schleischitz] \label{lemma2}
Let $\mathscr{C}$ be a curve as in \eqref{eq:curve} with $P_{j}\in{\mathbb{Z}[X]}$ 
as in \eqref{eq:polynome} of type 
$\underline{d}=(d_{1},\ldots,d_{k})$ and diameter $t\geq 1$. 
Further let $\zeta\in{\mathbb{R}}$ be arbitrary. For an integer $x$ denote by $y$ the 
closest integer to $\zeta x$ and write $y/x=y_{0}/x_{0}$ for integers $(x_{0},y_{0})=1$.

There exists a constant $C=C(\mathscr{C},\zeta)>0$ such that for any integer $x>0$ the estimate 
\begin{equation} \label{eq:hain}
\max_{1\leq j\leq k}\Vert P_{j}(\zeta)x\Vert < C\cdot x^{-t}
\end{equation}
implies $x_{1}^{d_{k}}$ divides $x$, where $x_{1}$ is defined 
via \eqref{eq:stock}, \eqref{eq:xeins} for $x_{0}$ as above.
A suitable choice for $C$ is given by 
\[
C=C_{0}:=\frac{1}{2D\cdot \Sigma(\mathscr{C},\zeta)},
\]
with $D=D(\mathscr{C})$ and $\Sigma(\mathscr{C},\zeta)$ from \eqref{eq:stock}.
\end{lemma}

The conditions $y_{0}/x_{0}=y/x$ and $x_{0}^{d_{k}}\vert x$ imply we may write
$(x,y)=(Mx_{0}^{d_{k}},Mx_{0}^{d_{k}-1}y_{0})$ for some integer $M$. This will be used the proofs in 
Section~\ref{beweischen}.
Observe also that in the case of monic polynomials $P_{j}$ we have $D=\Delta=1$ and $x_{0}=x_{1}$.

\section{New results}

\subsection{Lower bounds for Hausdorff dimensions}

The following is an extension of both~\cite[Lemma~1]{bug} and~\cite[Lemma~1.2]{schlei2} 
to the case of polynomials in arbitrarily many variables. The proof in Section~\ref{beweischen} uses a similar method.

\begin{lemma} \label{dadlemmane}
Let $s\geq 1$ be an integer and $P_{1},\ldots,P_{r}$ be polynomials $\mathbb{R}^{s}\to \mathbb{R}$ 
with rational coefficients.
Assume the maximum total degree among the $P_{j}$ is $d$ and let $k=s+r$.
Consider the manifold
\begin{equation} \label{eq:maddam}
M=\left\{(x_{1},x_{2},\ldots,x_{s},P_{1}(x_{1},\ldots,x_{s}),\ldots,P_{r}(x_{1},\ldots,x_{s}))\in{\mathbb{R}^{k}}:
(x_{1},\ldots,x_{s})\in{\mathbb{R}^{s}}\right\}.
\end{equation}
Then for all parameters $\lambda\geq 1/k$ we have 
\begin{equation} \label{eq:huigut}
\Pi_{s}(\mathscr{H}^{k}_{\lambda}\cap M) \supseteq  \mathscr{H}^{s}_{d\lambda+d-1}.
\end{equation}
\end{lemma}

With this result the dimension can be bounded with Jarn\'ik Theorem~\ref{jarnikchen}.

\begin{corollary} \label{dassaddas}
For any $M$ is in Lemma~\ref{dadlemmane} and any $\lambda\geq 1/s$ we have 
\begin{equation} \label{eq:dasser}
\dim(\mathscr{H}^{k}_{\lambda}\cap M) \geq \frac{s+1}{d(\lambda+1)}.
\end{equation}
\end{corollary}

\begin{proof}
The right hand side of \eqref{eq:huigut} attains the given dimension by Jarn\'ik Theorem~\ref{jarnikchen} 
and the projection $\Pi_{s}$ does
not affect Hausdorff dimensions since it is locally bi-Lipschitz.
\end{proof}

We emphasize that for $s=1$, where $M$ is a curve in $\mathbb{R}^{r+1}$, and sufficiently large parameters $\lambda$
depending on the degrees of the polynomials $P_{1},\ldots,P_{r}$, there is always equality
in \eqref{eq:huigut} and \eqref{eq:dasser}, by Theorem~\ref{jox}.

\subsection{Upper bounds for Hausdorff dimension} \label{lowbou}

In the general case it cannot be expected to have equality
in \eqref{eq:huigut} and \eqref{eq:dasser} for any large parameter $\lambda$ as for curves.
However, we will find some special cases where there is indeed equality in \eqref{eq:huigut} and \eqref{eq:dasser}
holds for $s>1$ and sufficiently large parameters $\lambda$, see Corollary~\ref{heilcor}. 

Generally, we notice the following upper bounds valid for more general manifolds, which follow rather 
immediately from Jarn\'ik~Theorem~\ref{jarnikchen}.

\begin{lemma} \label{struklemma}
Let
\[
M=\left\{(x_{1},x_{2},\ldots,x_{s},P_{1}(\underline{x}),\ldots,P_{r}(\underline{x}))\in{\mathbb{R}^{r+s}}:
(x_{1},\ldots,x_{s})\in{\mathbb{R}^{s}}\right\}
\]
where $\underline{x}=(x_{1},\ldots,x_{s})$ and $P_{j}:\mathbb{R}^{s}\to \mathbb{R}$ are polynomials
with real coefficients. Then
the dimension of $\mathscr{H}^{r+s}_{\lambda}\cap M$ is bounded above by 
$(s+1)(1+\lambda)^{-1}$ for any $\lambda \geq 1/s$. 
\end{lemma}

\begin{proof}
Indeed $(s+1)(1+\lambda)^{-1}$ is the correct dimension for the simultaneous $\lambda$-approximable points
of the first $s$ coordinates by Jarn\'ik Theorem~\ref{jarnikchen}. 
On the other hand, the inverse of the projection of the first $s$ coordinates onto $M$
is locally Lipschitz. Indeed, if in some compact set $K\subseteq \mathbb{R}^{s}$ 
we let 
\[
h(K):= \max_{1\leq j\leq r} \max_{1\leq i\leq s}\max_{u\in{K}} \left\vert \frac{dP_{j}}{dx_{i}}(u)\right\vert<\infty
\]
then clearly $\vert \underline{a}-\underline{b}\vert_{k,\infty} \leq h(K)
\cdot \vert \Pi_{s}(\underline{a})-\Pi_{s}(\underline{b})\vert_{s,\infty}$
for all $\underline{a},\underline{b}\in{\mathbb{R}^{r+s}}$ for which $\Pi_{s}(\underline{a}),\Pi_{s}(\underline{b})$
are both in $K$, where $\vert.\vert_{m,\infty}$ 
denotes the supremum norm in $\mathbb{R}^{m}$.
Finally asking for additional conditions for $P_{.}(x_{.})$ to be satisfied 
can only decrease this value.
\end{proof}

More generally, the same argument applies essentially to any locally Lipschitz maps 
$P_{j}:\mathbb{R}^{s}\to \mathbb{R}$.
Hence we are only interested in upper bounds smaller than $(s+1)(1+\lambda)^{-1}$. Note however that there
is equality in Lemma~\ref{struklemma} if the $P_{j}(\underline{x})$ all belong to the $\mathbb{Q}$-span
of $\{1,x_{1},\ldots,x_{s}\}$.

Now for the remainder of the paper we specialize to polynomials in only one variable.
We first display an immediate consequence of Theorem~\ref{jox}.

\begin{theorem} \label{dadthmne}
Let $s\geq 1$ be an integer and $\sigma_{1},\ldots,\sigma_{s}$ be positive integers. 
For any $1\leq i\leq s$ let $P_{i,1},\ldots,P_{i,\sigma_{1}}\in{\mathbb{Z}[X]}$ be a set of $\sigma_{i}$ polynomials 
in one variable of degrees $1=d_{i,1}\leq \ldots\leq d_{i,\sigma_{i}}$ respectively, where
$P_{i,1}(X)=X$ for $1\leq i\leq s$. Let
\[
t_{i}:=\max_{1\leq j\leq \sigma_{i}-1} d_{i,j+1}-d_{i,j}, \qquad 1\leq i\leq s.
\]
Put $k=\sigma_{1}+\cdots+\sigma_{s}$
and $t^{\prime}=\max_{1\leq i\leq s} t_{i}$. 
Let 
\[
M=\left\{(P_{1,1}(x_{1}),\ldots,P_{1,\sigma_{1}}(x_{1}),P_{2,1}(x_{2}),\ldots,P_{s,\sigma_{s}}(x_{s}))
\in{\mathbb{R}^{k}}:
(x_{1},\ldots,x_{s})\in{\mathbb{R}^{s}}\right\}.
\]
Denote $\Pi_{(s)}$ the projection on the coordinates 
$(1,\sigma_{1}+1,\sigma_{1}+\sigma_{2}+1,\ldots,\sigma_{1}+\cdots +\sigma_{s-1}+1)$,
which are the places with the entries $P_{i,1}(x_{i})=x_{i}$ in $M$.

Then for any parameter $\lambda>t^{\prime}$ we have
\[
\Pi_{(s)}(\mathscr{H}^{k}_{\lambda}\cap M)\subseteq 
\mathscr{H}^{1}_{d_{1,\sigma_{1}}\lambda+d_{1,\sigma_{1}}-1}\times 
\mathscr{H}^{1}_{d_{2,\sigma_{2}}\lambda+d_{2,\sigma_{2}}-1}\times 
\cdots\times \mathscr{H}^{1}_{d_{s,\sigma_{s}}\lambda+d_{s,\sigma_{s}}-1}.
\]
\end{theorem}

\begin{proof}
It suffices to look at the set of polynomials in each variable separately and apply Theorem~\ref{jox}.
\end{proof}

Assume that the Hausdorff dimension of direct products of sets $\mathscr{H}^{1}_{.}$ behaves nicely.

\begin{conjecture} \label{kondschektscha}
For any integer $s\geq 1$ and parameters $\eta_{1},\ldots,\eta_{s}$ at least one, we have 
\[
\dim(\mathscr{H}^{1}_{\eta_{1}}\times 
\mathscr{H}^{1}_{\eta_{2}}\times 
\cdots\times \mathscr{H}^{1}_{\eta_{s}})=\sum_{i=1}^{s} \dim(\mathscr{H}^{1}_{\eta_{i}})
= \sum_{i=1}^{s} \frac{2}{1+\eta_{i}}.
\]
\end{conjecture}

For products of general measurable sets the sum of the single dimensions is only known to be a lower bound, 
but we would need the reverse inequality. Usual covering arguments do not seem to be directly applicable. Also known relations for product sets to other dimension types
do not imply anything as the sets $\mathscr{H}^{1}_{.}$ are dense (the Hausdorff dimension of $A\times B$ is bounded above by the Hausdorff dimension of 
$A$ plus the upper packing dimension of $B$, see~\cite[p. 115]{mattila}. The upper packing dimension
in turn for subsets of a Euclidean space
coincides with the modified upper box dimension. Unluckily, the latter dimension is $1$ for 
a dense subset of $\mathbb{R}$). 

Provided Conjecture~\ref{kondschektscha} holds,
we can establish an upper bound for the Hausdorff dimensions in Theorem~\ref{dadthmne} which is
of certain interest, at least in many cases.

\begin{corollary}
Continue to use the notation and assumptions of Theorem~\ref{dadthmne}.
If Conjecture~\ref{kondschektscha} is true then for any parameter $\lambda>t^{\prime}$ the estimate
\begin{equation} \label{eq:widerlegung}
\dim(\mathscr{H}^{k}_{\lambda}\cap M)\leq \frac{2}{d_{1,\sigma_{1}}(\lambda+1)}+\cdots+
\frac{2}{d_{s,\sigma_{s}}(\lambda+1)}=\frac{2}{1+\lambda}\sum_{i=1}^{s} \frac{1}{d_{i,\sigma_{i}}}
\end{equation}
holds.
\end{corollary}

The quality of the bound obviously heavily depends on the degrees $d_{i,j}$. In case of $d_{i,\sigma_{i}}=1$ for all
$1\leq i\leq s$, that is $d_{i,j}=1$ for all pairs $i,j$,
then the resulting bound $2s/(1+\lambda)$ is worse than the trivial bound $(s+1)/(1+\lambda)$
from Lemma~\ref{struklemma}. However, in Remark~\ref{widerlrem} we will see that for other choices of degrees
the bound is reasonably good. See also Example~\ref{dadex}.

The results from Theorem~\ref{jox} can be extended with a bit more effort to compare the
considered sets with sets defined by {\em simultaneous} approximation properties in $\mathbb{R}^{s}$,
i.e. the sets $\mathscr{H}^{s}_{\lambda}$.
In the special case that the polynomials have equal degrees,
we will indeed establish equality in \eqref{eq:huigut} and \eqref{eq:dasser}. 
For the convenience of the reader
first we return to the case of only one polynomial in each variable (apart from the identity),
where the formulation and the proof are a bit less technical. Then we will present the more general version.  
We will directly apply Lemma~\ref{lemma2} instead of Theorem~\ref{jox}.
For the formulation of the next theorem
it will be convenient to deviate slightly from the notation of Theorem~\ref{dadthmne}
by rearranging the polynomials.

\begin{theorem} \label{oesterreich}
Let $s\geq 1$ be an integer and $k=2s$. For $1\leq i\leq s$ let $P_{i}(X)\in{\mathbb{Z}[X]}$ be any polynomial 
with integral coefficients in one variable of degree $d_{i}$. Let $d:=\min_{1\leq i\leq s} d_{i}$
and $d^{\prime}:= \max_{1\leq i\leq s} d_{i}$. Define
\[
M=\left\{(x_{1},x_{2},\ldots,x_{s},P_{1}(x_{1}),\ldots,P_{s}(x_{s}))\in{\mathbb{R}^{k}}:
(x_{1},\ldots,x_{s})\in{\mathbb{R}^{s}}\right\}.
\]
Then for any parameter $\lambda>\max\{d^{\prime}-1,1\}$ we have
\begin{equation} \label{eq:huigutt}
\Pi_{s}(\mathscr{H}^{k}_{\lambda}\cap M)\subseteq \mathscr{H}^{s}_{d\lambda+d-1}.
\end{equation}
In particular for $\lambda>\max\{d^{\prime}-1,1\}$ we have
\begin{equation} \label{eq:widerlegung2}
\dim(\mathscr{H}^{k}_{\lambda}\cap M)\leq \frac{s+1}{d(\lambda+1)}.
\end{equation}
\end{theorem}

\begin{remark}
The inclusions $
\Pi_{s}(\mathscr{H}^{k}_{\lambda}\cap M)\subseteq \Pi_{s}(\mathscr{H}^{k}_{\lambda})\subseteq \mathscr{H}^{s}_{\lambda}$
are trivial for any parameter $\lambda$. The rise in parameter of the right hand side in case of $d>1$ is essential.
\end{remark}

\begin{remark} \label{widerlrem}
We compare the bound \eqref{eq:widerlegung2} with the (conditioned) bound \eqref{eq:widerlegung}.
Without loss of generality assume $d=d_{1}\leq d_{2}\cdots\leq d_{s}=d^{\prime}$ in Theorem~\ref{oesterreich}.
Then the bounds in \eqref{eq:widerlegung} and \eqref{eq:widerlegung2} are
\[
\frac{2}{\lambda+1}\sum_{i=1}^{s} \frac{1}{d_{i}}, \qquad \frac{1}{d_{1}}\cdot \frac{s+1}{\lambda+1}
\]
respectively. So the bound in \eqref{eq:widerlegung2} is better if and only if
\[
2d_{1}\sum_{i=1}^{s} \frac{1}{d_{i}}\geq s+1.
\]
The left hand side is contained in $[2,\infty)$ and
the condition is clearly violated if $s>1$ and $d_{2}$ is much larger than $d_{1}$. Hence, assuming the conditioned bound
in \eqref{eq:widerlegung} is correct, in general there is no equality in \eqref{eq:widerlegung2}
for any large parameter, as for curves.  
\end{remark}

The combination of the lower and upper bounds gives equality in \eqref{eq:huigut} and \eqref{eq:dasser}
if we additionally assume that all polynomials have the same degree.

\begin{corollary} \label{heilcor}
We continue to use the definitions and notation of Theorem~\ref{oesterreich}.
For any parameter $\lambda>\max\{d^{\prime}-1,1\}$ in fact we have
\begin{equation} \label{eq:huitar}
\mathscr{H}^{s}_{d^{\prime}\lambda+d^{\prime}-1}\subseteq \Pi_{s}(\mathscr{H}^{k}_{\lambda}\cap M)
\subseteq \mathscr{H}^{s}_{d\lambda+d-1},
\end{equation}
and hence
\[
 \frac{s+1}{d^{\prime}(\lambda+1)}\leq \dim(\mathscr{H}^{k}_{\lambda}\cap M)\leq \frac{s+1}{d(\lambda+1)}.
\]
In particular, if $d^{\prime}=d$, that is
all of the polynomials have the same degree, then for $\lambda>\max\{d-1,1\}$ we have
\begin{equation} \label{eq:huiguttt}
\Pi_{s}(\mathscr{H}^{k}_{\lambda}\cap M)= \mathscr{H}^{s}_{d\lambda+d-1},
\end{equation}
and
\[
\dim(\mathscr{H}^{k}_{\lambda}\cap M)= \frac{s+1}{d(\lambda+1)}.
\]
\end{corollary}

\begin{proof}
The claim follows as a combination of Lemma~\ref{dadlemmane}, Theorem~\ref{oesterreich} 
and Jarn\'ik~Theorem~\ref{jarnikchen}.
\end{proof}

\begin{remark}
In case of equality $d=d^{\prime}$ the correct dimension $(s+1)d^{-1}(1+\lambda)^{-1}$ from Corollary~\ref{heilcor} 
indeed improves the (conditioned) upper bound in \eqref{eq:widerlegung}, which happens to be $2sd^{-1}(1+\lambda)^{-1}$
in the present notation, for $s>1$. Note also that the case of equality of all degrees 
and the resulting bound $2sd^{-1}(1+\lambda)^{-1}$ is the optimal upper bound one can get in \eqref{eq:widerlegung}
for fixed maximal degree $d$.
\end{remark}

As indicated, using Lemma~\ref{lemma2} in a more general form,
Theorem~\ref{oesterreich} can be extended to an arbitrary number of polynomials 
in every variable. This results a possibly larger range of parameters. The proof is very similar,
only the notation becomes slightly more demanding. We will again enumerate the polynomials
as in Theorem~\ref{dadthmne} for convenience.
 
\begin{theorem} \label{claimhold}
Let $s\geq 1$ be an integer and $\sigma_{1},\ldots,\sigma_{s}$ be positive integers. 
For any $1\leq i\leq s$ let $P_{i,1},\ldots,P_{i,\sigma_{1}}\in{\mathbb{Z}[X]}$ be a set of $\sigma_{i}$ polynomials 
in one variable of degrees $1=d_{i,1}\leq \ldots\leq d_{i,\sigma_{i}}$ respectively, where
$P_{i,1}(X)=X$ for $1\leq i\leq s$. Let
\[
t_{i}:=\max_{1\leq j\leq \sigma_{i}-1} d_{i,j+1}-d_{i,j}, \qquad 1\leq i\leq s.
\]
Put $k=\sigma_{1}+\cdots+\sigma_{s}$ and further let $d:=\min_{1\leq i\leq s} d_{i,\sigma_{i}}$
and $t^{\prime}:=\max_{1\leq i\leq s} t_{i}$. 
Let 
\[
M=\left\{(P_{1,1}(x_{1}),\ldots,P_{1,\sigma_{1}}(x_{1}),P_{2,1}(x_{2}),\ldots,P_{s,\sigma_{s}}(x_{s}))
\in{\mathbb{R}^{k}}:
(x_{1},\ldots,x_{s})\in{\mathbb{R}^{s}}\right\}.
\]

Denote $\Pi_{(s)}$ the projection on the coordinates 
$(1,\sigma_{1}+1,\sigma_{1}+\sigma_{2}+1,\ldots,\sigma_{1}+\cdots +\sigma_{s-1}+1)$,
which are the places with the entries $P_{i,1}(x_{i})=x_{i}$.
Then for any parameter $\lambda>t^{\prime}$ we have
\[
\Pi_{(s)}(\mathscr{H}^{k}_{\lambda}\cap M)\subseteq \mathscr{H}^{s}_{d\lambda+d-1},
\]
and thus
\[
\dim(\mathscr{H}^{k}_{\lambda}\cap M)\leq \frac{s+1}{d(\lambda+1)}.
\]
If $d=d^{\prime}$ there is again equality in both claims.
\end{theorem}

In the case that all $\sigma_{i}=2$ we again obtain the claim of Theorem~\ref{oesterreich},
as $t^{\prime}$ corresponds to $d^{\prime}-1$ (since $P_{i,1}$ is the identity so $d_{1,1}=1$) 
and the respective values of $d$ coincide. Still Theorem~\ref{claimhold} is only good if for every variable
there is at least one polynomial of rather large degree. We illustrate the results with three examples.

\begin{example}
Consider 
\[
M_{1}=\{(X,X^{3},X^{4},Y,Y^{6})\in{\mathbb{R}^{5}}: X,Y\in{\mathbb{R}}\}.
\]
Then $s=2, k=5, \sigma_{1}=3, \sigma_{2}=2$ and $d_{1,1}=1, d_{1,2}=3, d_{1,3}=4, d_{2,1}=1, d_{2,2}=6$
and $t_{1}=2, t_{2}=5$.
We conclude $d=\min\{4,6\}=4$ and $t^{\prime}=\max\{2,5\}=5$. Hence for $\lambda>5$, by Theorem~\ref{claimhold}
we have  
\[
\mathscr{H}^{2}_{6\lambda+5}\subseteq \Pi_{(2)}(\mathscr{H}^{5}_{\lambda}\cap M_{1})\subseteq \mathscr{H}^{2}_{4\lambda+3},
\]
with $\Pi_{(2)}$ the projection on the first and fourth coordinate, and thus 
\[
\frac{3}{6(1+\lambda)}\leq \dim(\mathscr{H}^{5}_{\lambda}\cap M_{1}) \leq \frac{3}{4(1+\lambda)}.
\]
The conditioned bound in \eqref{eq:widerlegung} would yield the weaker conclusion
\[
\dim(\mathscr{H}^{5}_{\lambda}\cap M_{1}) \leq \left(\frac{1}{3}+\frac{1}{6}\right)\frac{2}{(1+\lambda)}=
\frac{1}{\lambda+1}.
\]
\end{example}

\begin{example}
Consider 
\[
M_{2}=\{(X,X^{2},X^{3},X^{4},Y,Y^{2},Y^{3},Y^{4})\in{\mathbb{R}^{8}}: X,Y\in{\mathbb{R}}\},
\]
which can be regarded as the Cartesian product of two copies of the Veronese curve in dimension $4$.
Then $s=2, k=6$ and $d=d^{\prime}=4$ and $t^{\prime}=1$, thus by Theorem~\ref{claimhold}
\[
\dim(\mathscr{H}^{8}_{\lambda}\cap M_{2}) =\frac{3}{4(1+\lambda)}, \qquad \lambda>1.
\]
\end{example}

\begin{example} \label{dadex}
Let
\[
M_{3}=\{(X,X^{2},Y,Y^{5},Y^{7})\in{\mathbb{R}^{5}}: X,Y\in{\mathbb{R}}\}.
\]
Theorem~\ref{claimhold} implies
\[
\dim(\mathscr{H}^{5}_{\lambda}\cap M_{3})\leq \frac{3}{2(1+\lambda)}, \qquad \lambda>4.
\]
The conditioned result \eqref{eq:widerlegung} would yield the better upper bound
\[
\dim(\mathscr{H}^{5}_{\lambda}\cap M_{3})\leq \frac{2}{1+\lambda}\cdot 
\left(\frac{1}{2}+\frac{1}{7}\right)=\frac{9}{7}\cdot \frac{1}{1+\lambda}, \qquad \lambda>4.
\]
\end{example}

\section{Proofs} \label{beweischen}

\begin{proof} [Proof of Lemma~\ref{dadlemmane}]
We may assume the coefficients are integral. Let $\epsilon>0$ arbitrary but fixed.
For simplicity
let $\tau=d\lambda+d-1$.
Let $\underline{\zeta}=(\zeta_{1},\ldots,\zeta_{s})$ be any element of $\mathscr{H}^{s}_{\tau}$.
By definition of $\mathscr{H}^{s}_{\tau}$ the system
\[
\Vert q\zeta_{j}\Vert \leq q^{-\tau+\epsilon}, \qquad 1\leq j\leq s
\]
has arbitrarily large integer solutions $q$. We have
\begin{equation} \label{eq:taula}
\Vert q^{d}\zeta_{j}\Vert \leq q^{d-1}\Vert q\zeta^{j}\Vert \leq q^{-d\lambda+\epsilon}, \qquad 1\leq j\leq s.
\end{equation}
Let $Wx_{1}^{t_{1}}\cdots x_{s}^{t_{s}}$ be any
monomial of any of the polynomials $P_{1},\ldots,P_{r}$. 
By assumption $T:=t_{1}+\cdots+t_{r}\leq d$ and $W$ is an integer. Notice $\zeta_{i}^{t_{i}}\asymp 1$.
Without loss of generality assume $t_{1}>0$. First observe
\begin{equation} \label{eq:rororor}
\Vert q^{T}\cdot W\zeta_{1}^{t_{1}}\cdots\zeta_{s}^{t_{s}} \Vert 
= \Vert W(q^{t_{1}}\zeta_{1}^{t_{1}})\cdots (q^{t_{s}}\zeta_{s}^{t_{s}})\Vert
\leq W\Vert (q^{t_{1}}\zeta_{1}^{t_{1}})\cdots (q^{t_{s}}\zeta_{s}^{t_{s}})\Vert.  
\end{equation}
Recall the formula $X^{n}-Y^{n}=(X-Y)(X^{n-1}+\cdots+Y^{n-1})$. With $X=q\zeta_{i}$, 
$Y$ the closest integer to $X$ and $n=t_{i}$,
any expression $q^{t_{i}}\zeta_{i}^{t_{i}}$ can be written $A_{i}+b_{i}$ with $A_{i}\asymp q^{t_{i}}$ integers 
and $\vert b_{i}\vert\ll q^{-\tau+t_{i}-1+\epsilon}$, and $b_{i}=0$ if $t_{i}=0$. 
Let $C_{i}=A_{1}A_{2}\cdots A_{i-1}A_{i+1}\cdots A_{s}$. Expanding the product of 
$q^{t_{i}}\zeta_{i}^{t_{i}}=(A_{i}+b_{i})$
in the right hand side of \eqref{eq:rororor}, we infer
\[
\Vert q^{T}W\zeta_{1}^{t_{1}}\cdots\zeta_{s}^{t_{s}} \Vert \ll
\max_{1\leq i\leq s} b_{i}C_{i}\ll q^{-\tau+t_{i}-1+\epsilon}q^{T-t_{i}}=q^{-\tau+T-1+\epsilon}.  
\]
Since $T\leq d$ we infer
\[
\Vert q^{d}W\zeta_{1}^{t_{1}}\cdots\zeta_{s}^{t_{s}} \Vert 
\leq Wq^{d-T}\Vert q^{T}\zeta_{1}^{t_{1}}\cdots\zeta_{s}^{t_{s}} \Vert 
\ll q^{-\tau+d-1+\epsilon}\ll q^{-d\lambda+\epsilon}.  
\]
This estimate holds for all monomials.
Thus if we let $\zeta_{s+i}:=P_{i}(\underline{\zeta})$ for $1\leq i\leq r$ then
\begin{equation} \label{eq:ergaenzt}
\Vert q^{d}\zeta_{s+i}\Vert=
\Vert q^{d}P_{i}(\underline{\zeta})\Vert \ll q^{-d\lambda+\epsilon}, \qquad 1\leq i\leq r.  
\end{equation}
Combination of \eqref{eq:taula} and \eqref{eq:ergaenzt} gives that the element 
$\tilde{\zeta}:=\Pi_{s}^{-1}(\underline{\zeta})\cap M=(\zeta_{1},\ldots,\zeta_{k})$ 
lies in $\mathscr{H}^{k}_{\lambda-\epsilon/d}\cap M$, and $\Pi_{s}(\tilde{\zeta})=\underline{\zeta}$. 
Since $\underline{\zeta}$ was arbitrary in $\mathscr{H}^{s}_{\tau}$ 
the claim follows as $\epsilon$ can be chosen arbitrarily small.
\end{proof}

We turn to the upper bounds.

\begin{proof} [Proof of Theorem~\ref{oesterreich}]
For $d^{\prime}=1$ the polynomials are linear and the claim follows elementarily, so let us assume $d^{\prime}>1$.
Let $\epsilon>0$ be arbitrary but fixed.
Assume $\underline{\zeta}=(\zeta_{1},\ldots,\zeta_{k})$ is in the left hand side of \eqref{eq:huigutt}, such that
\begin{align}
\vert q\zeta_{i}-p_{i}\vert &\leq q^{-\lambda+\epsilon}, \qquad 1\leq i\leq s  \label{eq:schweren} \\
\vert qP_{i}(\zeta_{i})-p_{s+i}\vert &\leq q^{-\lambda+\epsilon}, \qquad 1\leq i\leq s  \label{eq:oeter}
\end{align}
holds for arbitrarily large integers $q$ and integers $p_{1},\ldots,p_{2s}$.
First assume additionally that all $P_{j}$ are monic.
Consider the systems \eqref{eq:schweren}, \eqref{eq:oeter} for any fixed $i$.
Notice that $d_{i}-1$ coincides with $t$ from Lemma~\ref{lemma2} in the $i$-th system.
Since $\lambda>d^{\prime}-1\geq d_{i}-1$, we may assume that $\epsilon$ is sufficiently small
that also $\lambda-\epsilon>d_{i}-1$.
It follows from Lemma~\ref{lemma2} (where $x$ corresponds to the present $q$) that
any large solution $(p_{i},p_{s+i},q)$ of such a system for fixed $1\leq i\leq s$ satisfies 
$(p_{i},q)=(M_{i}q_{i}^{d_{i}-1}y_{i},M_{i}q_{i}^{d_{i}})$ with integers $q_{i}, M_{i},y_{i}$.
Consequently the identity
\begin{equation} \label{eq:fritzdiek}
\vert q\zeta_{i}-p_{i}\vert= M_{i}q_{i}^{d_{i}-1} \vert q_{i}\zeta_{i}-y_{i}\vert, \qquad 1\leq i\leq s,
\end{equation}
holds. Let $f$ be the lowest common multiple of the $q_{i}, 1\leq i\leq s$. Then $q_{i}\vert f$ for
all $i$. On the other hand, since $q_{i}^{d_{i}}\vert q$, also the lowest common multiple of the integers 
$q_{i}^{d_{i}}$ divides $q$. In particular the lowest common multiple of the integers $q_{i}^{d}$
divides $q$. Since taking a fixed power of any element of a set 
and taking lowest common multiples of the set commutes,
the considered lowest common multiple coincides with $f^{d}$. Thus we have just proved $f^{d}\vert q$.
Write $rf^{d}=q$ for an integer $r$. Moreover, as $q_{i}\vert f$, we may write
$f=r_{i}q_{i}$ for $1\leq i\leq s$ and integers $r_{i}$. Combining these equations we
can write $q=(rr_{i}^{d})q_{i}^{d}$. To sum up we have
\begin{equation} \label{eq:atzatz}
\frac{q}{f}=rf^{d-1}=rr_{i}^{d-1}q_{i}^{d-1}=\frac{M_{i}q_{i}^{d_{i}-1}}{r_{i}}. 
\end{equation}
With \eqref{eq:fritzdiek} and \eqref{eq:atzatz} we compute
\[
\vert f\zeta_{i}-r_{i}y_{i}\vert = r_{i}\vert q_{i}\zeta_{i}-y_{i}\vert
= \frac{r_{i}}{M_{i}q_{i}^{d_{i}-1}} \vert q\zeta_{i}-p_{i}\vert=\frac{f}{q}\vert q\zeta_{i}-p_{i}\vert
=\frac{1}{rf^{d-1}}\vert q\zeta_{i}-p_{i}\vert, \qquad 1\leq i\leq s.
\]
Thus
\begin{equation} \label{eq:romans}
-\frac{\log \Vert f\zeta_{i}\Vert}{\log f}=
-\frac{\log \Vert q\zeta_{i}\Vert-(d-1)\log f-\log r}{\log f}
\geq -\frac{\log \Vert q\zeta_{i}\Vert}{\log f}+d-1.
\end{equation}
However, the first expression on the right hand side can be estimated with \eqref{eq:schweren} 
and $q=rf^{d}\geq f^{d}$, which yields
\[
-\frac{\log \Vert q\zeta_{i}\Vert}{\log f}=-\frac{\log \Vert q\zeta_{i}\Vert}{\log q}\cdot \frac{\log q}{\log f}
\geq (\lambda-\epsilon)\cdot d.
\]
Inserting this in \eqref{eq:romans} we end up with 
\[
-\frac{\log \Vert f\zeta_{i}\Vert}{\log f} \geq d\lambda+d-1-\epsilon d, \qquad 1\leq i\leq s.
\]
Since the values of $f$ obviously tend to infinity as $q$ does and $\epsilon$ can be made arbitrarily small,
this is equivalent to the fact that 
$(\zeta_{1},\ldots,\zeta_{s})=\Pi_{(s)}(\underline{\zeta})$ belongs to the set $\mathscr{H}^{s}_{d\lambda+d-1}$.
Since $\underline{\zeta}$ was an arbitrary element of $\mathscr{H}^{k}_{\lambda}\cap M$ the proof 
of the case where all $P_{j}$ are monic is complete. 

In the general case, i.e. if the polynomials are not necessarily monic, one may proceed very similarly,
where $q_{i}$ has to be slightly adapted in case of $\Delta\neq 1$ in view of Lemma~\ref{lemma2}.
However, since the denominator in \eqref{eq:xeins} is bounded by $\Delta$ which depends on $M$ only,
this is not of much relevance. We omit the solely technical proof.
\end{proof}

\begin{proof} [Proof of Theorem~\ref{claimhold}]
We basically follow the proof of Theorem~\ref{oesterreich}.
Consider for any $1\leq i\leq s$ more generally the system
\[
\vert qP_{j}(\zeta_{i})-p_{i,j}\vert \leq q^{-\lambda+\epsilon}, \qquad 1\leq j\leq \sigma_{i}. 
\]
Again it follows from Lemma~\ref{lemma2} that for $\lambda>t^{\prime}\geq t_{i}$ and small $\epsilon>0$  
the solution vectors satisfy
$(p_{i,1},q)=(M_{i}q_{i}^{d_{i,\sigma_{i}}-1}y_{i},M_{i}q_{i}^{d_{i,\sigma_{i}}})$
for integers $M_{i},q_{i},y_{i}$. Again
if $f$ is the lowest common multiple of the $q_{i}^{d_{i,\sigma_{i}}}$, 
we may write $r_{i}q_{i}=f$ and $f^{d}\vert q$ by the same argument as in Theorem~\ref{oesterreich}.
The rest of the proof is identical to that of Theorem~\ref{oesterreich}.
\end{proof}

\end{document}